\documentclass[11pt,letterpaper,reqno]{amsart}

\usepackage{amssymb}
\usepackage{amsmath}
\usepackage{amsthm}
\usepackage{amsfonts}
\usepackage{bbm}
\usepackage{graphicx}
\usepackage{doi}

\usepackage{bookmark}
\usepackage{hyperref}
\hypersetup{pdfstartview={FitH}}

\newtheorem{theorem}{Theorem}

\newtheorem{proposition}[theorem]{Proposition}

\theoremstyle{definition}
\newtheorem{remark}[theorem]{Remark}

\numberwithin{equation}{section}

\renewcommand{\leq}{\leqslant}
\renewcommand{\geq}{\geqslant}

\begin{document}

\title[Large dilates of hypercube graphs in the plane]{Large dilates of hypercube graphs in the plane}

\author[V. Kova\v{c}]{Vjekoslav Kova\v{c}}
\author[B. Predojevi\'{c}]{Bruno Predojevi\'{c}}

\address{Department of Mathematics, Faculty of Science, University of Zagreb, Bijeni\v{c}ka cesta 30, 10000 Zagreb, Croatia}

\email{vjekovac@math.hr}
\email{bruno.predojevic@math.hr}



\begin{abstract}
We study a distance graph $\Gamma_n$ that is isomorphic to the $1$-skeleton of an $n$-dimensional unit hypercube. We show that every measurable set of positive upper Banach density in the plane contains all sufficiently large dilates of $\Gamma_n$. This provides the first examples of distance graphs other than the trees for which a dimensionally sharp embedding in positive density sets is known.
\end{abstract}

\maketitle


\section{Introduction}
One line of investigation in geometric measure theory was initiated with a question posed by Sz\'{e}kely \cite{Sze83} (and popularized by Erd\H{o}s \cite{Erd83:open}), who asked if a planar set of positive upper density (see Section~\ref{sec:notation} for definitions of the usual densities) realizes all sufficiently large distances between pairs of its points.
This question has been answered affirmatively by Furstenberg, Katznelson, and Weiss \cite{FKW90:dist}, and independently also by Falconer and Marstrand \cite{FM86:dist} and by Bourgain \cite{B86:roth}. 
Each of the three proofs turned out quite influential; for instance the approach of Falconer and Marstrand \cite{FM86:dist} has been generalized to ``very dense'' sets in $\mathbb{R}^d$ \cite{FKY22}. Bourgain's approach \cite{B86:roth} made the greatest impact and it triggered a series of results studying more general rigid configurations in $\mathbb{R}^d$ \cite{LM16:prod,LM19:hypergraphs,K20:anisotrop,DS22}, not necessarily with respect to the Euclidean distance \cite{Kol04,CMP15:roth,DKR18,DK21,DK22}.

Results of this type are often shown in sufficiently large dimensions only, as higher dimensions add ``degrees of freedom'' and make it easier to identify the desired pattern.
The ``slicing argument'' rigorously justifies this reasoning, so the results are sometimes formulated only in the minimal number of dimensions $d$ in which they are known to hold.
However, it is usually still unresolved what this minimal dimension $d$, needed to identify sufficiently large dilates of a given configuration in every set of positive upper density, actually is.

A large class of ``flexible'' configurations was covered by Lyall and Magyar \cite{LM20} in their study of embeddings of the so-called distance graphs which they define as finite connected graphs whose set of vertices is contained in some Euclidean space. Furthermore, they define a distance graph $\Gamma$ to be $k$-degenerate, for some positive integer $k$, if every induced subgraph of $\Gamma$ contains a vertex of degree at most $k$. The smallest such $k$ is known as the \emph{degeneracy} of $\Gamma$. Their main result, \cite[Theorem 2]{LM20}, requires the dimensional threshold $d\geq k+1$ for every $k$-degenerate distance graph they discussed (which they called proper distance graphs). In other words, for a certain subclass of $k$-degenerate distance graphs $\Gamma$ they showed that every set $A\subseteq\mathbb{R}^{k+1}$ of positive upper density contains an isometric copy of the dilate $\lambda\cdot\Gamma$ for all sufficiently large numbers $\lambda>0$.
One can easily check that a distance graph is a tree if and only if it is $1$-degenerate, so it follows from their result that trees can be embedded in the minimal possible dimension $d=2$. 
(An alternative approach to trees can be found in \cite[Section~5]{K20:anisotrop}.) However it is natural to ask if there are other distance graphs, besides trees, for which such a dimensionally optimal embedding is possible. Similarly, one can ask whether there exist distance graphs of arbitrarily large degeneracy whose large dilates can still be embedded in the plane $\mathbb{R}^2$. The motivation for this paper lies in providing examples of such graphs.

We are interested in the distance graph $\Gamma_n$ that is (isomorphic to) a $1$-skeleton of an $n$-dimensional unit hypercube. Notice that $\Gamma_n$ is $n$-degenerate.
Unfolding the more general terminology from \cite{LM20}, one says that a set $A\subseteq\mathbb{R}^2$ \emph{contains an isometric copy of} $\lambda\cdot\Gamma_n$ for some number $\lambda>0$ if there exist a point $x\in\mathbb{R}^2$ and vectors $y_1,\ldots,y_n\in\mathbb{R}^2$ of Euclidean length $\lambda$ such that
\begin{equation}\label{eq:allpoints}
x + r_1 y_1 + \cdots + r_n y_n \in A 
\end{equation}
for all $2^n$ tuples $(r_1,\ldots,r_n)\in\{0,1\}^n$ and that all $2^n$ points in \eqref{eq:allpoints} are mutually distinct; see Figure~\ref{fig:graph1}.
Here is the precise formulation of our first result.

\begin{figure}
\includegraphics[width=0.5\linewidth]{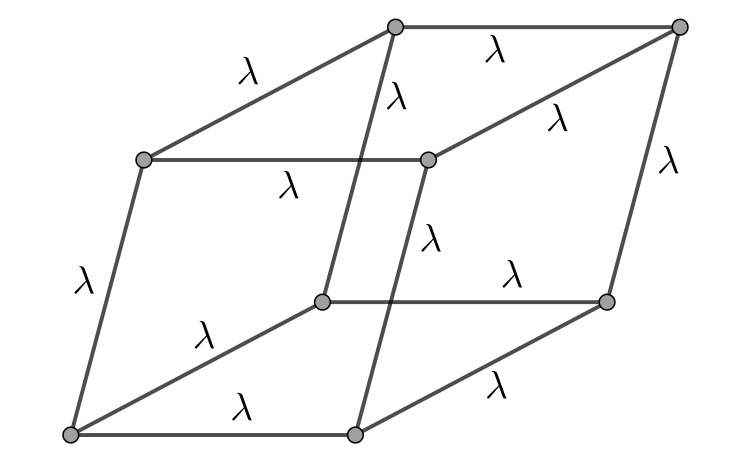}
\caption{Distance graph $\lambda\cdot\Gamma_3$.}
\label{fig:graph1}
\end{figure}

\begin{theorem}\label{thm:large}
For a positive integer $n$ and a measurable set of positive upper Banach density $A\subseteq\mathbb{R}^2$ there exists a number $\lambda_0(A,n)>0$ such that for every number $\lambda\geq\lambda_0(A,n)$ the set $A$ contains an isometric copy of the distance graph $\lambda\cdot\Gamma_n$.
\end{theorem}

There is some history of variants of Theorem~\ref{thm:large} if one is content with embedding $\lambda\cdot\Gamma_n$ in higher-dimensional Euclidean spaces.

\begin{itemize}
\item Groundbreaking work by Lyall and Magyar \cite{LM16:prod} developed a possible approach to product configurations $\Delta_1\times\Delta_2$ and, in particular, showed that sufficiently large dilates of rigid squares can be found in positive density subsets of $\mathbb{R}^4$; see \cite[Theorem 1.2]{LM16:prod}.
\item Higher-dimensional hypercubes were a bit more difficult and the same result for rigid $n$-hypercubes when $n\geq3$ follows from the studies of more general configurations, done independently by Durcik and Kova\v{c} in $\mathbb{R}^{5n}$ \cite[Theorem 1]{DK21} and by Lyall and Magyar in $\mathbb{R}^{2n}$ \cite[Theorem 1.1(i)]{LM19:hypergraphs}.
\item The most general rigid configuration for which a result like Theorem~\ref{thm:large} is known is due to Lyall and Magyar \cite[Theorem 1.2(i)]{LM19:hypergraphs}; it embeds an $n$-fold product $\Delta_1\times\cdots\times\Delta_n$ of vertex-sets of non-degenerate simplices in positive density subsets of $\mathbb{R}^{\mathop{\textup{card}}(\Delta_1)+\cdots+\mathop{\textup{card}}(\Delta_n)}$.
\item Additional improvements are possible if one gives up the rigidity of the hypercube and starts ``flexing'' its $1$-skeleton. The aforementioned work by Lyall and Magyar \cite{LM20} recognizes $\Gamma_n$ as a proper $n$-degenerate distance graph, so it has already been known prior to this paper that a positive upper Banach density subset of $\mathbb{R}^{n+1}$ contains an isometric copy of $\lambda\cdot\Gamma_n$ at all sufficiently large scales $\lambda$ \cite[Theorem 2(i)]{LM20}.
\item Theorem~\ref{thm:large} above embeds $\lambda\cdot\Gamma_n$ in $\mathbb{R}^2$, which is clearly dimensionally optimal. Already Sz\'{e}kely's question (which is the case $n=1$) obviously has a negative answer in $\mathbb{R}$: just consider $A=[-1/10,1/10]+\mathbb{Z}$.
\end{itemize}

There has also been interest in ``compact versions'' (in the language of Bourgain \cite{B86:roth}) of Euclidean density theorems. The following weaker but quantitative variant of Theorem~\ref{thm:large} will follow the same line of proof.
Let $\vert A\vert $ denote the Lebesgue measure of $A\subseteq\mathbb{R}^2$.

\begin{theorem}\label{thm:interval}
For a positive integer $n$ there exists a positive constant $C(n)$ with the following property: for every $0<\delta\leq1/2$ and every measurable set $A\subseteq[0,1]^2$ satisfying $\vert A\vert \geq\delta$ there is an interval $I(A,n)\subseteq(0,1]$ of length at least $\exp(-\delta^{-C(n)})$ such that for every $\lambda\in I(A,n)$ the set $A$ contains an isometric copy of the distance graph $\lambda\cdot\Gamma_n$.
\end{theorem}

Weaker variants of Theorem~\ref{thm:interval} in higher dimensions can also be deduced from the existing literature in numerous ways.
\begin{itemize}
\item Lyall and Magyar \cite[Theorem 2(ii)]{LM20} showed a general quantitative result of this type for proper distance graphs, but there $\lambda\cdot\Gamma_n$ needs to be embedded in $[0,1]^{n+1}$.
\item For completely rigid hypercubes of side-length $\lambda$, Lyall and Magyar \cite[Theorem 1.1(ii)]{LM19:hypergraphs} showed an analogous result in $[0,1]^{2n}$, but with a tower-exponential bound on the length of the interval $I(A,n)$. Then they proceeded to study multiple products of vertex-sets of simplices $\Delta_1\times\cdots\times\Delta_n$ \cite[Theorem 1.2(ii)]{LM19:hypergraphs}.
\item Durcik and Kova\v{c} \cite[Theorem 3]{DK22} sharpened the last result for hypercubes in $[0,1]^{2n}$ to the same type of bound for the length of $I(A,n)$ as in Theorem~\ref{thm:interval}, namely a single exponential in a negative power of $\delta$.
\item Durcik and Stip\v{c}i\'{c} \cite[Theorem 1]{DS22} extended the aforementioned quantitatively reasonable result from \cite{DK22} to multiple products of vertex-sets of simplices, simultaneously generalizing \cite[Theorem 1.2(ii)]{LM19:hypergraphs} and \cite[Theorem 3]{DK22}.
\item Theorem~\ref{thm:interval} is also clearly dimensionally optimal for embeddings of $\Gamma_n$: in $[0,1]$ one can simply consider $A=[0,\varepsilon]\cup[3\varepsilon,4\varepsilon]\cup[6\varepsilon,7\varepsilon]\cup\cdots$, which has measure about $1/3$, while $\varepsilon>0$ can be arbitrarily small.
\end{itemize}

\begin{figure}
\includegraphics[width=0.43\linewidth]{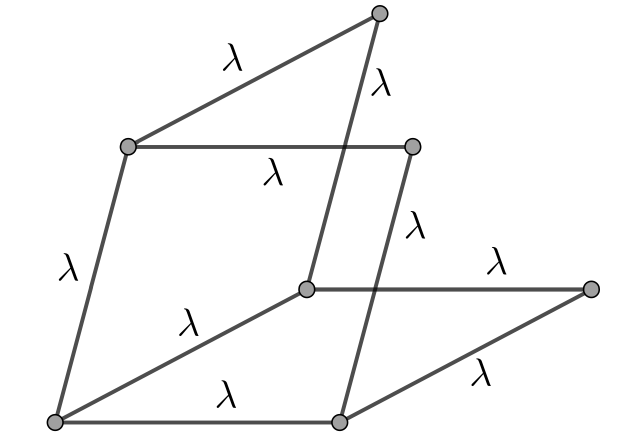}
\caption{A subgraph of $\lambda\cdot\Gamma_3$.}
\label{fig:graph2}
\end{figure}

Certain subgraphs of $\Gamma_n$ (or their minor modifications) have appeared in the work of Fitzpatrick, Iosevich, McDonald, and Wyman \cite[Section~4]{FIMW21} (see the collection of adjoined rhombi in Figure~\ref{fig:graph2}) albeit in the context of configurations in two-dimensional vector spaces over finite fields. This fact provided us with another source of motivation for formulating Theorems~\ref{thm:large} and \ref{thm:interval}.

\begin{remark}
All of the aforementioned results in the literature relevant to embeddings of large dilates of $\Gamma_n$ also apply to the $1$-skeleton of a rectangular box with edge lengths $c_1,\ldots,c_n\in(0,\infty)$, viewed as a distance graph.
This is also the case with our Theorems~\ref{thm:large} and \ref{thm:interval}, which we formulated in the case $c_1=\cdots=c_n=1$ only, not to overwhelm the notation. For instance, purely cosmetic changes in the proof of the first theorem can also show that the $1$-skeleton of a rectangular box with edge lengths $a_j=\lambda c_j$, $j=1,\ldots,n$, can be embedded in a positive upper Banach density set $A\subseteq\mathbb{R}^2$ for all sufficiently large numbers $\lambda$ depending on $A$; see Figure~\ref{fig:graph3}.
In the other direction, each prescribed edge length of the distance graph that we are embedding clearly needs to be sufficiently large, as the case $a_1=1$ (or, in fact, any other fixed number) is easily prohibited by constructing a set $A$ of positive upper density that avoids unit distances between its points.
Quite interestingly, a certain coloring constructed in \cite{Kov23color} to answer a seemingly unrelated question of Erd\H{o}s and Graham \cite[\#189]{EP} (also see \cite[p.~331]{EG79} and \cite[p.~56]{GB15}) can be modified to show that there exists a positive upper density planar set $A$ in which we cannot embed any $1$-skeleton of a rectangular box with sides $a_j$ satisfying 
\begin{equation}\label{eq:productrelation}
a_1 a_2 \cdots a_n=1; 
\end{equation}
the details are given in \cite{Kov23color}.
This difficulty disappears in higher dimensions.
The fact that even a single non-compact relation between edge lengths, like \eqref{eq:productrelation}, is problematic for embeddings in $\mathbb{R}^2$ adds to the subtleties of Theorems~\ref{thm:large} and \ref{thm:interval}.
\end{remark}

\begin{figure}
\includegraphics[width=0.5\linewidth]{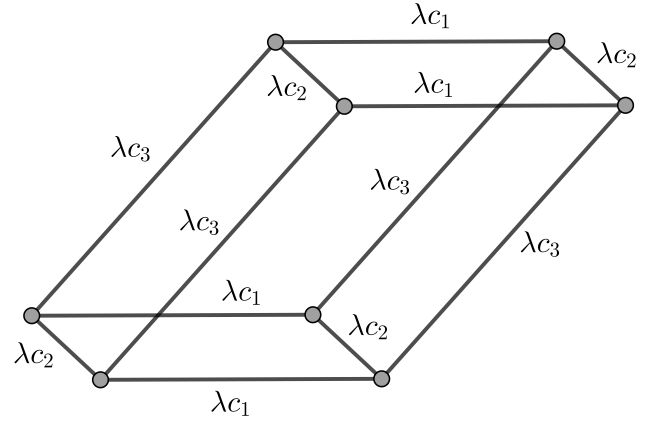}
\caption{A variant of $\lambda\cdot\Gamma_3$.}
\label{fig:graph3}
\end{figure}

A crucial ingredient in the proofs of Theorems~\ref{thm:large} and \ref{thm:interval} is a bound for the multilinear singular integral forms associated with the hypercube graphs. 
We will include a self-contained treatment of those analytical objects to the extent needed for the geometric application in this paper, but let us also briefly comment on them in higher generality. 
The singular integral form that will appear in Section~\ref{sec:error} (see \eqref{eq:theform}) is a particular case of a \emph{singular Brascamp--Lieb form}, namely 
\begin{equation}\label{eq:BLgeneral}
\Lambda(f_1,\ldots,f_m) = \mathop{\textup{p.v.}}\int_{\mathbb{R}^D} \Big( \prod_{j=1}^{m} f_j (\Pi_j\mathbf{x}) \Big) K(\Pi\mathbf{x}) \,\textup{d}\mathbf{x}, 
\end{equation}
where $\Pi_1,\ldots,\Pi_m,\Pi$ are surjective linear maps from $\mathbb{R}^D$ to lower-dimensional Euclidean spaces, while $K$ is a smooth Calder\'{o}n-Zygmund kernel.
The term \emph{singular Brascamp--Lieb inequality} was first used by Durcik and Thiele \cite{DT20} for any $\textup{L}^p$ estimate for the form \eqref{eq:BLgeneral}. These inequalities are significantly more difficult than their non-singular counterparts (with kernel $K$ omitted) and the literature is very far from their complete theory; see the survey paper \cite{DT21survey} by Durcik and Thiele. No estimates are known already for seemingly simple instances, such as the so-called triangular Hilbert transform \cite{KTZ15}.
However, the particular ``cubical'' case, when there are $2^n$ maps $\Pi_j$ and they are projections of the form
\[ \mathbb{R}^{2n}\to\mathbb{R}^n,\quad (x_1^0, \ldots, x_n^0, x_1^1,\ldots, x_n^1) \mapsto (x_1^{r_1}, \ldots, x_n^{r_n}) \]
for $(r_1,\ldots,r_n)\in\{0,1\}^n$, has seen much progress over the last ten years \cite{Kov12,Kov11,Dur15,Dur17,DT20} (even before they were called the singular Brascamp--Lieb forms). The current state-of-the-art paper is the one by Durcik, Slav\'{\i}kov\'{a}, and Thiele \cite{DST22}.
Particular instances of the singular cubical forms have already found applications in geometric measure theory \cite{DKR18,DK21,DK22,DS22}, probability \cite{KS15,KS20}, and ergodic theory \cite{Kov16,DKST19}.
The form \eqref{eq:theform} needed in Section~\ref{sec:error} is not of the cubical type, but in the case $d=1$ it can be reduced to the studied cubical form by composing the functions $f_j$ with certain skew-projections, while the cases $d\geq2$ would be quite analogous.
Thus, we could have essentially just invoked the result of \cite{DST22} in the second half of Section~\ref{sec:error}. However, in order to make the proofs of Theorems~\ref{thm:large} and \ref{thm:interval} more elementary, we prefer to give a short self-contained proof of the particular estimate that we need, namely, \eqref{eq:singspec} below.

Proofs of Theorems~\ref{thm:large} and \ref{thm:interval} will respectively follow the approaches from \cite[Section~4]{K20:anisotrop} and \cite[Section~7]{DK22}. 
However, some technicalities need to be done differently for a flexible configuration and in the plane. The most notable differences are the appearance of the Gowers' uniformity norms $\Vert \cdot\Vert_{\textup{U}^n}$ in the study of the so-called structured part of the counting form in Section~\ref{sec:structured} (and not when studying the so-called uniform part in Section~\ref{sec:uniform}), and the appearance of the aforementioned hypercube graph singular integral forms in Section~\ref{sec:error}.

Density theorems in geometric measure theory have gained a lot of attention in the literature over the last ten years. The reader can consult \cite{Kovac:survey} for a very brief survey of the very recent developments.


\section{Notation, basic definitions, and counting forms}
\label{sec:notation}
We will write $A\lesssim_P B$ and $B\gtrsim_P A$ if the inequality $A\leq C_P B$ holds for some unimportant finite constant $C_P$ depending only on a set of parameters $P$. We also write $A \sim_P B$ if both $A\lesssim_P B$ and {  $B\lesssim_P A$} hold.
Sometimes it is understood that the constant $C_P$ also depends on a few paramaters that are not explicitly listed in the set $P$, and which will be fixed throughout the proof. In our case this will be the pattern dimension $n$.

We will write $\mathbbm{1}_{S}$ for the indicator function (i.e., the characteristic function) of a set $S\subseteq\mathbb{R}^2$.
The imaginary unit will be denoted by $\mathbbm{i}$. The Euclidean norm of a vector $v$ will be written as $\vert v\vert $, while the standard inner product of $v$ and $w$ will be denoted $v\cdot w$.
We will write $\textup{B}_r(x)$ for the Euclidean ball of radius $r>0$ centered at $x\in\mathbb{R}^2$.

The \emph{upper Banach density} of a measurable set $A$ in the Euclidean plane $\mathbb{R}^2$ is defined to be
\[ \overline{\delta}(A)
:= \limsup_{R\to\infty} \sup_{x\in\mathbb{R}^2} \frac{\vert A\cap(x+[0,R]^2)\vert }{R^2}, \]
where $\vert B\vert $ denotes the two-dimensional Lebesgue measure of $B$.
It is clearly greater than or equal to the more common (\emph{centered}) \emph{upper density} of a measurable $A\subseteq\mathbb{R}^2$, given by
\[ \limsup_{R\to\infty} \frac{\vert A\cap[-R/2,R/2]^2\vert }{R^2}. \]
Note that $\overline{\delta}(A)$ is a more flexible quantity, in the sense that it also detects accumulated ``mass'' on arbitrary translates of large squares $[-R/2,R/2]^2$.

Fix an integer $n\geq1$. Let $\mathcal{B}(\mathbb{R}^D)$ denote the Borel sigma algebra on $\mathbb{R}^D$.
We write $\sigma$ for the normalized spherical measure on $\mathbb{S}^{1}\subset\mathbb{R}^2$. 

Let $\mathbbm{g}$ be the \emph{standard Gaussian} on $\mathbb{R}^2$, $\mathbbm{h}^{(i)}$ its partial derivatives, and $\mathbbm{k}$ its Laplacian, i.e.,
\[ \mathbbm{g}(x) := e^{-\pi\vert x\vert ^2},\quad
\mathbbm{h}^{(i)} := \partial_i \mathbbm{g}, \ i=1,2,\quad
\mathbbm{k} := \Delta \mathbbm{g}. \]
Using the usual normalization for the Fourier transform,
\[ \widehat{f}(\xi) := \int_{\mathbb{R}^2} f(x) e^{-2\pi \mathbbm{i} x\cdot\xi} \,\textup{d}x, \]
we easily get
\begin{equation}\label{eq:FtofGauss}
\widehat{\mathbbm{g}}(\xi) = e^{-\pi\vert \xi\vert ^2},
\quad \widehat{\mathbbm{h}^{(i)}}(\xi) = 2\pi \mathbbm{i}\xi_i e^{-\pi\vert \xi\vert ^2},
\quad \widehat{\mathbbm{k}}(\xi) = -4\pi^2 \vert \xi\vert ^2 e^{-\pi\vert \xi\vert ^2}.
\end{equation}
In general, we denote by $\tau_\lambda$ and $f_\lambda$ dilates of a measure $\tau$ on $\mathcal{B}(\mathbb{R}^D)$ and an integrable function $f$ on $\mathbb{R}^D$ by a factor $\lambda>0$, respectively defined as
\[ \tau_{\lambda}(E) := \tau\Big(\frac{1}{\lambda}E\Big), \quad { f_{\lambda}}(x) := \frac{1}{\lambda^D} f\Big(\frac{1}{\lambda}x\Big). \]
Their convolution is, on the other hand, defined as a function $\tau\ast f$ on $\mathbb{R}^D$ given by
\[ (\tau\ast f)(x) := \int_{\mathbb{R}^D} f(x-y) \,\textup{d}\tau(y). \]
Using \eqref{eq:FtofGauss} and the basic identity for the Fourier transform of a convolution, we easily obtain
\begin{align}
\sum_{l=1}^2 \mathbbm{h}^{(l)}_{\alpha} \ast \mathbbm{h}^{(l)}_{\beta} &=\frac{\alpha \beta}{\alpha^2 + \beta ^2} \mathbbm{k}_{\sqrt{\alpha^2 + \beta^2}} , \label{conv par der Gauss} \\
\mathbbm{k}_{\alpha} \ast \mathbbm{g}_{\beta} &=\frac{\alpha^2}{\alpha^2 + \beta^2} \mathbbm{k}_{\sqrt{\alpha^2 + \beta^2}} \label{conv Gauss Laplace}
\end{align}
for { every} $\alpha, \beta >0$.
Dilates of $\mathbbm{g}$ and $\mathbbm{k}$ satisfy the (re-parameterized) \emph{heat equation},
\begin{equation}\label{eq:heateq}
\frac{\partial}{\partial t} \mathbbm{g}_t(x) = \frac{1}{2\pi t} \mathbbm{k}_t(x)
\end{equation}
for $(t,x)\in(0,\infty)\times\mathbb{R}^2$, which is easily seen by straightforward differentiation.

Let $f\colon\mathbb{R}^2\to[0,1]$ be a compactly supported measurable function.
It is convenient to denote
\[ \mathcal{F}_n(x;y_1,\ldots,y_n) := \prod_{(r_1,\ldots,r_n)\in\{0,1\}^n} f(x + r_1 y_1 + \cdots + r_n y_n) \]
for $x,y_1,\ldots,y_n\in\mathbb{R}^2$.
For instance, in the case $n=2$ we have
\[ \mathcal{F}_2(x;y_1,y_2) = f(x) f(x+y_1) f(x+y_2) f(x+y_1+y_2). \]
Throughout this paper we will always use this notation in the special case when $f$ is equal to an indicator function of some set. These expressions are well-known from additive combinatorics, because their integrals are powers of the \emph{Gowers uniformity norms} $\Vert \cdot\Vert_{\textup{U}^n}$, namely,
\begin{equation}\label{eq:Gowersdef}
\int_{(\mathbb{R}^2)^{n+1}} \mathcal{F}_n(x;y_1,\ldots,y_n) \,\textup{d}x \,\textup{d}y_1 \cdots \textup{d}y_n = \Vert f\Vert_{\textup{U}^n(\mathbb{R}^2)}^{2^n}.
\end{equation}
Note that $\mathcal{F}_n(x;y_1,\ldots,y_n)$ is symmetric in $y_1,\ldots,y_n$ and that it satisfies the recurrence relation:
\begin{equation}\label{eq:rekurzija}
\mathcal{F}_n(x;y_1,\ldots,y_n) = \mathcal{F}_{n-1}(x;y_1,\ldots,y_{n-1}) \mathcal{F}_{n-1}(x+y_n;y_1,\ldots,y_{n-1}).
\end{equation}
The well-known \emph{Gowers--Cauchy--Schwarz inequality} \cite{Gow01,HK05,ET12} on the group $\mathbb{R}^2$ reads
\begin{align*}
\bigg\vert  \int_{(\mathbb{R}^2)^{n+1}} \prod_{(r_1,\ldots,r_n)\in\{0,1\}^n} f_{r_1,\ldots,r_n}(x + r_1 y_1 + \cdots + r_n y_n) \,\textup{d}x \,\textup{d}y_1 \cdots \textup{d}y_n \bigg\vert  & \\
\leq \prod_{(r_1,\ldots,r_n)\in\{0,1\}^n} \Vert f_{r_1,\ldots,r_n}\Vert_{\textup{U}^n(\mathbb{R}^2)} & .
\end{align*}
for bounded, compactly supported, measurable, { real-valued} functions $f_{r_1,\ldots,r_n}$.
It is merely a consequence of many applications of the Cauchy--Schwarz inequality.
Moreover, if $f$ is nonnegative, bounded, measurable, and supported on a cube $Q\subset\mathbb{R}^d$, then we have
\begin{equation}\label{eq:GCScorr}
\Vert f\Vert_{\textup{U}^n(\mathbb{R}^2)} \gtrsim_{n} \vert Q\vert ^{-1+(n+1)/2^n} \int_Q f .
\end{equation}
To see this, we simply apply the Gowers--Cauchy--Schwarz inequality when $f_{0,\ldots,0}=f$ and all other $f_{r_1,\ldots,r_n}$ are equal to the characteristic function of the cube $\widetilde{Q}$ obtained by dilating $Q$ from its center by the factor of $n+1$. Namely, let $Q_0$ be the cube congruent to $Q$, but translated so that its center coincides with the origin.
If $x\in Q$ and $y_1,\ldots,y_n\in Q_0$, then $x + r_1 y_1 + \cdots + r_n y_n\in\widetilde{Q}$ for any choice of $r_1,\ldots,r_n\in\{0,1\}$.

Let us finally define an appropriate \emph{counting form} by
\[ \mathcal{N}^{0}_{\lambda}(f) :=
\int_{\mathbb{R}^{2}} \int_{(\mathbb{R}^2)^n} \mathcal{F}_n(x;y_1,\ldots,y_n) \,\textup{d}\sigma_{\lambda}^{\otimes n}(y_1,\ldots,y_n) \,\textup{d}x \]
for every $\lambda>0$.
Moreover, we will also need the smoothed out version of $\mathcal{N}_{\lambda}^{0}$, defined as
\[ \mathcal{N}^{\varepsilon}_{\lambda}(f) :=
\int_{(\mathbb{R}^2)^{n+1}} \mathcal{F}_n(x;y_1,\ldots,y_n) \prod_{k=1}^n (\sigma_{\lambda} \ast \mathbbm{g}_{\varepsilon \lambda})(y_k) \,\textup{d}y_1 \cdots \textup{d}y_n \,\textup{d}x \]
for $\lambda>0$ and $0<\varepsilon\leq1$.
Note that we have
\begin{equation}\label{eq:Nconvergence}
\lim_{\varepsilon\to0+} \mathcal{N}_{\lambda}^{\varepsilon}(f) = \mathcal{N}_{\lambda}^{0}(f).
\end{equation}
This can be verified by first observing that
\[ \mathcal{G}(y_1,\ldots,y_n) := \int_{\mathbb{R}^2} \mathcal{F}_n(x;y_1,\ldots,y_n) \,\textup{d}x \]
is a continuous function on $(\mathbb{R}^2)^n$, merely by the continuity of translation operators on $\textup{L}^1(\mathbb{R}^2)$. Then \eqref{eq:Nconvergence} follows simply by rewriting
\begin{align*}
\mathcal{N}^{\varepsilon}_{\lambda}(f) & = \int_{(\mathbb{R}^2)^n} (\mathcal{G}\ast\mathbbm{g}^{\otimes n}_{\varepsilon\lambda})(y_1,\ldots,y_n)\,\textup{d}\sigma_{\lambda}^{\otimes n}(y_1,\ldots,y_n), \\ \mathcal{N}^{0}_{\lambda}(f) & = \int_{(\mathbb{R}^2)^n} \mathcal{G}(y_1,\ldots,y_n)\,\textup{d}\sigma_{\lambda}^{\otimes n}(y_1,\ldots,y_n) 
\end{align*}
and using the pointwise limit $\lim_{\varepsilon\to0+}\mathcal{G}\ast\mathbbm{g}^{\otimes n}_{\varepsilon\lambda}=\mathcal{G}$ and the dominated convergence theorem with respect to the measure $\textup{d}\sigma_{\lambda}^{\otimes n}(y_1,\ldots,y_n)$.

Using the scheme of approach inaugurated by Cook, Magyar, and Pramanik \cite{CMP15:roth}, the main idea is to decompose
\begin{equation}\label{eq:decomposition}
\mathcal{N}^{0}_{\lambda}(f) = \mathcal{N}^{1}_{\lambda}(f) + \big(\mathcal{N}^{\varepsilon}_{\lambda}(f) - \mathcal{N}^{1}_{\lambda}(f)\big) + \big(\mathcal{N}^{0}_{\lambda}(f) - \mathcal{N}^{\varepsilon}_{\lambda}(f)\big)
\end{equation}
for an appropriately chosen $\varepsilon\in(0,1]$.
The three summands on the right hand side of \eqref{eq:decomposition} are respectively called the \emph{structured part}, the \emph{error part}, and the \emph{uniform part}; see the ``philosophical'' discussion in \cite{K20:anisotrop}.
Proofs of Theorems~\ref{thm:large} and \ref{thm:interval} proceed by controlling the structured part from below and bounding the other two parts from above, modulo the fact that pigeonholing in the scales $\lambda$ is used in the discussion of the error part.


\begin{remark}\label{rem:nondegenerate}
It is not immediately clear that the positivity of the counting form $\mathcal{N}_{\lambda}^{0}(f)$ implies the existence of a nondegenerate hypercube configuration. In principle, some of the vertices may overlap. However, such degenerate cases are negligible. To see this, let $S,T\subseteq \{1, \ldots, n\}$, $S\neq T$, for some natural number $n$. Define
\[ H_{S,T} := \bigg\{(y_1, \ldots , y_n) \in (\mathbb{R}^2)^n \, : \, \sum_{i\in S}y_i = \sum_{i\in T}y_i \bigg\}. \]
Finally set
\[ H=\bigcup_{\substack{S,T\subseteq \{1, \ldots, n\}\\S\neq T}} H_{S,T}. \]
For arbitrary fixed $S,T$ as above consider the \emph{degenerate counting form}
\begin{align*}
\mathcal{N}^{0,S,T}_{\lambda}(f) :&=
\int_{\mathbb{R}^{2}} \int_{H_{S,T}} \mathcal{F}_n(x;y_1,\ldots,y_n) \,\textup{d}\sigma_{\lambda}^{\otimes n}(y_1,\ldots,y_n) \,\textup{d}x.
\end{align*}
There exist $1\leq j\leq n$ and $s_1,s_2,\ldots,s_{j-1}\in\{-1,0,1\}$ (all depending on $S,T$) such that the hyperplane $H_{S,T}$ can be written as 
\[ H_{S,T} = \{(y_1, \ldots , y_n) \in (\mathbb{R}^2)^n \, : \, y_j = s_1 y_1 + s_2 y_2 + \cdots + s_{j-1} y_{j-1}\}, \]
so, consequently, 
\begin{align*}
\mathcal{N}^{0,S,T}_{\lambda}(f) 
= \int_{\mathbb{R}^{2}} \int_{(\mathbb{R}^2)^{j-1}} \int_{\{s_1 y_1 + s_2 y_2 + \cdots + s_{j-1} y_{j-1}\}} \int_{(\mathbb{R}^2)^{n-j}} \mathcal{F}_n(x;y_1,\ldots,y_n) & \\
\textup{d}\sigma_{\lambda}^{\otimes (n-j)}(y_{j+1},\ldots,y_n) \,\textup{d}\sigma_{\lambda}(y_j) \,\textup{d}\sigma_{\lambda}^{\otimes (j-1)}(y_1,\ldots,y_{j-1}) \,\textup{d}x & = 0,
\end{align*}
since the integral over the singleton $\{s_1 y_1 + s_2 y_2 + \cdots + s_{j-1} y_{j-1}\}$ with respect to the measure $\sigma_{\lambda}$ certainly equals $0$. Therefore, 
\begin{align*}
\int_{\mathbb{R}^{2}} \int_{(\mathbb{R}^2)^n\setminus H} \mathcal{F}_n(x;y_1,\ldots,y_n) \,\textup{d}\sigma_{\lambda}^{\otimes n}(y_1,\ldots,y_n) \,\textup{d}x & \\
\geq \mathcal{N}^{0}_{\lambda}(f) - \sum_{\substack{S,T\subseteq \{1, \ldots, n\}\\S\neq T}}\mathcal{N}^{0,S,T}_{\lambda}(f) = \mathcal{N}^{0}_{\lambda}(f) & .
\end{align*}
Once we show that $\mathcal{N}^{0}_{\lambda}(\mathbbm{1}_A)>0$, it will follow that there exist $x\in\mathbb{R}^2$ and $(y_1,\ldots,y_n)\in(\mathbb{R}^2)^n\setminus H$ such that $\mathcal{F}_n(x;y_1,\ldots,y_n)>0$, which means that all $2^n$ points \eqref{eq:allpoints} really belong to $A$ and are mutually distinct.
We conclude that $\mathcal{N}^{0}_{\lambda}(\mathbbm{1}_A)>0$ is sufficient to guarantee that $A$ contains an isometric copy of $\lambda\cdot\Gamma_n$.
\end{remark}

Now we are ready to begin the proofs of Theorems \ref{thm:large} and \ref{thm:interval}, which occupy the rest of the paper.
As we have already mentioned, the strategy is to control separately each of the three terms from the decomposition \eqref{eq:decomposition}.


\section{The structured part}
\label{sec:structured}

\begin{proposition}
For numbers $R\geq\lambda>0$ and a measurable set $B\subseteq[0,R]^2$ we have
\begin{equation}\label{eq:structured}
\mathcal{N}^{1}_{\lambda}(\mathbbm{1}_B) \geq c_{\textup{str}} \Big(\frac{\vert B\vert }{R^2}\Big)^{2^n} R^{2},
\end{equation}
where $c_{\textup{str}}\in(0,\infty)$ is a constant that depends only on $n$.
\end{proposition}

\begin{proof}
Recall that
\[ \mathcal{N}^{1}_{\lambda}(\mathbbm{1}_B) =
\int_{(\mathbb{R}^2)^{n+1}} \mathcal{F}_n(x;y_1,\ldots,y_n) \prod_{k=1}^n (\sigma_{\lambda} \ast \mathbbm{g}_{\lambda})(y_k) \,\textup{d}y_1 \cdots \textup{d}y_n \,\textup{d}x .\]
On the right hand side we use an easy pointwise estimate,
\[ \sigma\ast\mathbbm{g} \gtrsim \mathbbm{1}_{\textup{B}_2(0)} \geq \mathbbm{1}_{[-1,1]^2}; \]
see \cite[Section~3.1]{K20:anisotrop} for a slightly more general formulation. That way we obtain
\[
\mathcal{N}^{1}_{\lambda}(\mathbbm{1}_B) \gtrsim_n
\lambda^{-2n} \int_{(\mathbb{R}^2)^{n+1}} \mathcal{F}_n(x;y_1,\ldots,y_n) \prod_{k=1}^n \mathbbm{1}_{[-\lambda,\lambda]^2}(y_k) \,\textup{d}y_1 \cdots \textup{d}y_n \,\textup{d}x.
\]
We partition $[0,R]^2$ into squares of side length $\lambda$. More precisely, we consider the set 
\[ \mathcal{Q}:=\bigg\{ [(k-1)\lambda,k\lambda) \times [(l-1)\lambda,l\lambda) \, : \, k,l\in\Big\{1,2, \dots, \Big\lceil\frac{R}{\lambda}\Big\rceil\Big\} \bigg\}, \]
which covers $[0,R]^2$ completely. Notice that
$\vert \mathcal{Q}\vert =\lceil R/\lambda\rceil^2\leq (2R/\lambda)^2$ 
and that each square $Q \in \mathcal{Q}$ has area $\lambda^2$. We impose further restrictions on the above integral by requiring that all the $2^n$ vertices of the degenerate paralelotope \eqref{eq:allpoints} lie within the same element of $\mathcal{Q}$, that is, due to the fact that the elements of $\mathcal{Q}$ are mutually disjoint, we can write
\begin{align*}
\mathcal{N}^{1}_{\lambda}(\mathbbm{1}_B) \gtrsim_n \lambda^{-2n}
\sum_{Q\in\mathcal{Q}}\int_{(\mathbb{R}^2)^{n+1}} \prod_{(r_1,\dots,r_n)\in\{0,1\}^n}\mathbbm{1}_{B\cap Q}(x+r_1y_1+\cdots+r_ny_n) & \\
\prod_{k=1}^n \mathbbm{1}_{[-\lambda,\lambda]^2}(y_k) \,\textup{d}y_1 \cdots \textup{d}y_n \,\textup{d}x & .
\end{align*}
Also note that if $x+r_1y_1+\cdots+r_ny_n$ all lie within the same square of side-length $\lambda$, it trivially follows that all $y_k$ belong to $[-\lambda,\lambda]^2$. Therefore, the last display becomes simply
\begin{align*}
\mathcal{N}^{1}_{\lambda}(\mathbbm{1}_B) \gtrsim_n \lambda^{-2n}
\sum_{Q\in\mathcal{Q}}\int_{(\mathbb{R}^2)^{n+1}} \prod_{(r_1,\dots,r_n)\in\{0,1\}^n} & \mathbbm{1}_{B\cap Q}(x+r_1y_1+\cdots+r_ny_n) \,\textup{d}y_1 \cdots \textup{d}y_n \,\textup{d}x,
\end{align*}
i.e., by the definition of the Gowers norms \eqref{eq:Gowersdef},
\[ \mathcal{N}^{1}_{\lambda}(\mathbbm{1}_B) \gtrsim_n \lambda^{-2n}
\sum_{Q\in\mathcal{Q}} \Vert \mathbbm{1}_{B\cap Q}\Vert_{\textup{U}^n(\mathbb{R}^d)}^{2^n}. \]
Using \eqref{eq:GCScorr} with $f=\mathbbm{1}_{B\cap Q}$ followed by Jensen's inequality, we obtain
\begin{align*}
\mathcal{N}^{1}_{\lambda}(\mathbbm{1}_B) & \gtrsim_n \lambda^{-2n}
\sum_{Q\in\mathcal{Q}} \vert Q\vert ^{-2^n+n+1} \Big(\int_{Q} \mathbbm{1}_{B\cap Q}\Big)^{2^n} 
= \lambda^{-2^{n+1}+2} \sum_{Q\in\mathcal{Q}} \vert B\cap Q\vert ^{2^n} \\
& \geq \lambda^{-2^{n+1}+2} \vert \mathcal{Q}\vert ^{-2^n+1} \Big(\sum_{Q\in\mathcal{Q}} \vert B\cap Q\vert  \Big)^{2^n}
= \lambda^{-2^{n+1}+2} \vert \mathcal{Q}\vert ^{-2^n+1} \vert B\vert ^{2^n} \\
& \gtrsim_n \lambda^{-2^{n+1}+2} \Big(\frac{R}{\lambda}\Big)^{-2^{n+1}+2} \vert B\vert ^{2^n}
= \Big(\frac{\vert B\vert }{R^2}\Big)^{2^n} R^{2}.
\end{align*}
This completes the proof of \eqref{eq:structured}.
\end{proof}


\section{The error part}
\label{sec:error}

\begin{proposition}
For every positive integer $J$, {  for all ``scales''} $0<\lambda_1<\cdots<\lambda_J$ satisfying $\lambda_{j+1}\geq2\lambda_j$ for $j=1,\ldots,J-1$, every number $R\geq2\lambda_J$, every measurable set $B\subseteq[0,R]^{2}$, and every $\varepsilon\in(0,1]$ we have the estimate:
\begin{equation}\label{eq:error}
\sum_{j=1}^{J} \big\vert \mathcal{N}^{\varepsilon}_{\lambda_j}(\mathbbm{1}_B)-\mathcal{N}^{1}_{\lambda_j}(\mathbbm{1}_B)\big\vert  \leq C_{\textup{err}}\varepsilon^{-3n}\log\Big(\frac{1}{\varepsilon}\Big) R^{2},
\end{equation}
where $C_{\textup{err}}\in(0,\infty)$ is some constant that depends only on $n$.
\end{proposition}

\begin{proof}
Write $f=\mathbbm{1}_{B}$. Differentiating the product and using the heat equation \eqref{eq:heateq} we get
\begin{equation}\label{eq:productrule}
\frac{\partial}{\partial t} \prod_{k=1}^n(\sigma_{\lambda} \ast \mathbbm{g}_{t\lambda})(y_k)
= \frac{1}{2\pi t} \sum_{m=1}^{n} (\sigma_{\lambda} \ast \mathbbm{k}_{t\lambda})(y_m)
\Big( \prod_{\substack{1\leq k\leq n\\ k\neq m}} (\sigma_{\lambda} \ast \mathbbm{g}_{t\lambda})(y_k) \Big).
\end{equation}
Therefore, for $0<\alpha<\beta\leq1$ we can write
\begin{equation*}
\mathcal{N}^{\alpha}_{\lambda}(f)-\mathcal{N}^{\beta}_{\lambda}(f) = \sum_{m=1}^{n} \mathcal{L}_{\lambda}^{\alpha,\beta,m}(f),
\end{equation*}
where
\begin{align*}
\mathcal{L}_{\lambda}^{\alpha,\beta,m}(f) :=
-\frac{1}{2\pi} \int_{\alpha}^{\beta} 
 \int_{(\mathbb{R}^2)^{n+1}} 
 &\mathcal{F}_n(x;y_1,\ldots,y_n) \\
& (\sigma_{\lambda} \ast \mathbbm{k}_{t \lambda})(y_m) \prod_{\substack{1\leq k\leq n\\ k\neq m}}(\sigma_{\lambda} \ast \mathbbm{g}_{t\lambda})(y_k) 
\,\textup{d}y_1 \cdots \textup{d}y_n \,\textup{d}x \,\frac{\textup{d}t}{t}.
\end{align*}
That way we obtain the bound
\begin{equation}\label{error to bound}
\sum_{j=1}^{J} \big\vert \mathcal{N}^{\varepsilon}_{\lambda_j}(\mathbbm{1}_B)-\mathcal{N}^{1}_{\lambda_j}(\mathbbm{1}_B)\big\vert  \lesssim  \sum_{m=1}^{n} \sum_{j=1}^{J} \big\vert \mathcal{L}_{\lambda_j}^{\varepsilon,1,m}(\mathbbm{1}_B) \big\vert  .
\end{equation}
Due to the symmetry of the variables $y_1,\dots, y_n$ in the previous expressions, it is sufficient to work with the last summand (i.e., the one for $m=n$) on the right hand side of \eqref{error to bound}.
Let $\theta= 10^{-1}e^{-1}$ and observe that, for any $\lambda, t >0$, 
\[
\int_{\theta t \lambda}^{e \theta t \lambda} \frac{\textup{d}s}{s}=1 .
\]
For $j \in \lbrace 1,2, \ldots J\rbrace$, $t \in [\varepsilon,1]$, and $s \in [\theta t \lambda_j, e \theta t \lambda_j]$, set 
\[ r=r(j,t,s):=\sqrt{(t \lambda_j)^2-2s^2}. \]
A simple calculation shows that $r\sim s \sim t\lambda_j$.
Using the Gaussian convolution identities \eqref{conv par der Gauss} and \eqref{conv Gauss Laplace}, we can rewrite  
\begin{equation}\label{eq:erraux1}
\mathbbm {k}_{t\lambda_j}=\frac{(t\lambda_j)^2}{s^2}\sum_{l=1}^2 \mathbbm{h}^{(l)}_{s} \ast \mathbbm{h}^{(l)}_{s} \ast \mathbbm{g}_{r}.
\end{equation}
The following Gaussian inequalities appear as \cite[(4.5)]{K20:anisotrop} and \cite[(4.6)]{K20:anisotrop}, respectively:
\begin{align}
(\sigma_{\lambda_j}\ast\mathbbm{g}_{t\lambda_j})(x) & \lesssim \varepsilon^{-3} \int_1^{\infty} \mathbbm{g}_{s\gamma}(x) \,\frac{\textup{d}\gamma}{\gamma^2},
\label{eq:erraux2} \\
(\sigma_{\lambda_j}\ast\mathbbm{g}_{r})(x) & \lesssim \varepsilon^{-3} \int_1^{\infty} \mathbbm{g}_{s\gamma}(x) \,\frac{\textup{d}\gamma}{\gamma^2}.
\label{eq:erraux3}
\end{align}
Using the identity \eqref{eq:erraux1} and expanding out the convolution, using the inequalities \eqref{eq:erraux2} and \eqref{eq:erraux3}, recalling the recursive identity \eqref{eq:rekurzija} and the fact that $\mathbbm{h}^{(l)}_s$ is odd, and introducing the change of variables $u=y_n+x$, we obtain that
\[ \sum_{j=1}^{J} \big\vert \mathcal{L}_{\lambda_j}^{\varepsilon,1,n}(\mathbbm{1}_B) \big\vert  \]
is at most
\begin{align*}
\varepsilon^{-3n}\sum_{l=1}^2 \sum _{j=1}^J \int_{[1,\infty)^n}\int_{\varepsilon}^1\int_{\theta t \lambda_j}^{e \theta t \lambda_j}\int_{(\mathbb{R}^2)^{n+1}}
&\Big\vert  \int_{\mathbb{R}^2}\mathcal{F}_{n-1}(u;y_1,\ldots,y_{n-1}) \mathbbm{h}^{(l)}_{s}(u-z_1) \,\textup{d}u \Big\vert \\
&\Big\vert  \int_{\mathbb{R}^2}  \mathcal{F}_{n-1}(x;y_1,\ldots,y_{n-1}) \mathbbm{h}^{(l)}_{s}(x-z_2)  \,\textup{d}x \Big\vert  \\
& \mathbbm{g}_{s \gamma_n} (z_1 -z_2)  
\prod_{1\leq k\leq n-1} \mathbbm{g}_{s \gamma_k}(y_k) \\
& \textup{d}z_1\,\textup{d}z_2\,\textup{d}y_1 \cdots \textup{d}y_{n-1} \,\frac{\textup{d}s}{s} \,\frac{\textup{d}t}{t}\,\frac{\textup{d}\gamma_1}{\gamma_1^2} \cdots\,\frac{\textup{d}\gamma_n}{\gamma_n^2} .
\end{align*}
The expression above can be treated in the same way as a similar expression in \cite[Subsection 4.2]{K20:anisotrop}, namely, we apply the Cauchy--Schwarz inequality to separate the product of functions $\mathcal{F}_{n-1}$ and then we notice that after a change of variables we can multiply the two square roots to obtain
\begin{align*}
\varepsilon^{-3n} \sum_{l=1}^2 \sum _{j=1}^J \int_{[1,\infty)^n}\int_{\varepsilon}^1\int_{\theta t \lambda_j}^{e \theta t \lambda_j}\int_{(\mathbb{R}^2)^{n+1}}
\Big(& \int_{\mathbb{R}^2}  \mathcal{F}_{n-1}(u;y_1,\ldots,y_{n-1}) \mathbbm{h}^{(l)}_{s}(u-z_1)\,\textup{d}u \Big)^2 \\
&\mathbbm{g}_{s \gamma_n} (z_1 -z_2) \prod_{1\leq k\leq n-1} \mathbbm{g}_{s \gamma_k}(y_k)  \\
&\textup{d}z_1\,\textup{d}z_2\,\textup{d}y_1 \cdots \textup{d}y_{n-1} \,\frac{\textup{d}s}{s} \,\frac{\textup{d}t}{t}\,\frac{\textup{d}\gamma_1}{\gamma_1^2} \cdots\,\frac{\textup{d}\gamma_n}{\gamma_n^2} .
\end{align*}
Using the translation invariance of the Lebesgue measure, we integrate out the Gaussian dilated by $s \gamma_n$ in the variable $z_2$ and then we integrate out $1/\gamma_n^{2}$. Next, we expand the square term inside the integral, and finally we undo the previously introduced change of variables and use the recursive identity \eqref{eq:rekurzija}.
This will allow us to notice a convolution of the partial derivatives of the Gaussian, which we can rewrite as a  dilated Gaussian Laplacian $\mathbbm{k}$ using identity \eqref{conv par der Gauss}  again. Finally, we are in a position to eliminate the scales $\lambda_j$ out of the consideration.
We observe that due to the fact that the scales $\lambda_j$ grow for at least a factor of $2$, for a fixed $t \in [\varepsilon, 1]$, each $s\in(0,\infty )$ belongs to at most two of the intervals $[\theta t \lambda_j, e \theta t \lambda_j]$, therefore, we can bound the previous by
\begin{align*}
-\varepsilon^{-3n} \int_{[1,\infty)^{n-1}}\int_{\varepsilon}^1\int_{0}^{\infty}\int_{(\mathbb{R}^2)^{n+1}}
&  \mathcal{F}_{n}(x;y_1,\ldots,y_{n}) \mathbbm{k}_{\sqrt{2}s}(y_n) 
\\
& \prod_{1\leq k\leq n-1}\mathbbm{g}_{s \gamma_k}(y_k) 
 \,\textup{d}x\,\textup{d}y_1 \cdots \textup{d}y_{n} \,\frac{\textup{d}s}{s} \,\frac{\textup{d}t}{t}\,\frac{\textup{d}\gamma_1}{\gamma_1^2} \cdots\,\frac{\textup{d}\gamma_{n-1}}{\gamma_{n-1}^2} . 
\end{align*}
Integrating with respect to $t$ further yields the factor of $\log(1/\varepsilon)$.

To complete the proof, we need to bound { the above} integral form, which is similar to \cite[(2.11)]{K20:anisotrop}. It is precisely { a particular} case of the singular Brascamp--Lieb form \eqref{eq:BLgeneral} mentioned in the introduction, { with the singular kernel given by
\[ K(y_1,\ldots,y_n) = \int_{0}^{\infty} \mathbbm{k}_{\sqrt{2}s}(y_n)  \prod_{1\leq k\leq n-1}\mathbbm{g}_{s \gamma_k}(y_k) \,\frac{\textup{d}s}{s} . \]
As we have already said, instead of invoking and twisting the existing literature, we will rather reuse a few tricks from \cite{Kov12,Dur15} needed to complete the proof in this section.}

For $n \in \mathbb{N}$, $m \in \{1, \ldots, n \}$, scales $\gamma_1,\ldots, \gamma_n \in ( 0 ,\infty )$ and a compactly supported $f \in \textup{L}^{2^n}(\mathbb{R}^2)$, let
\begin{align}
\Theta_{\gamma_1, \ldots, \gamma_n}^{n,m}(f):=-\int_0^{\infty}\int_{(\mathbb{R}^2)^{n+1}}
& \mathcal{F}_n(x;y_1,\ldots, y_n) \mathbbm{k}_{s\gamma_m}(y_m) \nonumber \\
& \Big(\prod_{\substack{1 \leq k \leq n \\ k \neq m }}\mathbbm{g}_{s\gamma_k}(y_k)\Big)
\,\textup{d}x\,\textup{d}y_1\cdots\textup{d}y_n \, \frac{\textup{d}s}{s}. \label{eq:theform}
\end{align}
Using the convolution identity \eqref{conv par der Gauss} it follows that
\[
\mathbbm{k}_{s\gamma_m}=2\sum_{l=1}^2\mathbbm{h}_{s\gamma_m/\sqrt{2}}^{(l)} \ast \mathbbm{h}_{s\gamma_m/\sqrt{2}}^{(l)} .
\]
For fixed $y_1, \dots, y_n$ and $m \in \{1, \dots, n \}$, let 
\[ F_m(u):=\mathcal{F}_{n-1}(u;y_1,\dots,y_{m-1},y_{m+1},\dots,y_n). \]
Using a recursive relation similar to \eqref{eq:rekurzija}, making the change of variables $u=x+y_n$ and using the above identity, we can write
\begin{align*}
\Theta_{\gamma_1, \ldots ,\gamma_n}^{n,m}(f)=-2\sum_{l=1}^2\int_0^{\infty}\int_{(\mathbb{R}^2)^{n-1}}
&\left\langle F_m \ast\mathbbm{h}_{s\gamma_m/\sqrt{2}}^{(l)} \ast \mathbbm{h}_{s\gamma_m/\sqrt{2}}^{(l)}, F_m\right\rangle_{\textup{L}^2(\mathbb{R}^2)}\\
&\Big(\prod_{\substack{1 \leq k \leq n \\ k \neq m }}\mathbbm{g}_{s\gamma_k}(y_k)\Big)
\,\textup{d}y_1\cdots\,\textup{d}y_{m-1}\,\textup{d}y_{m+1}\cdots\textup{d}y_n \, \frac{\textup{d}s}{s}.
\end{align*}
A simple calculation shows that
\[
\left\Vert F_m \ast\mathbbm{h}_{s\gamma_m/\sqrt{2}}^{(l)}\right\Vert_{\textup{L}^2(\mathbb{R}^2)}^2=-
\left\langle F_m \ast\mathbbm{h}_{s\gamma_m/\sqrt{2}}^{(l)} \ast \mathbbm{h}_{s\gamma_m/\sqrt{2}}^{(l)}, F_m\right\rangle_{\textup{L}^2(\mathbb{R}^2)} ,
\]
from which follows that $\Theta_{\gamma_1, \ldots, \gamma_n}^{n,m}(f)$ is well defined and nonnegative.
Using the product formula \eqref{eq:productrule} again
and the fundamental theorem of calculus, we obtain
\begin{align*}
\sum_{m=1}^n\Theta_{\gamma_1, \ldots, \gamma_n}^{n,m}(f) &= 2\pi\lim_{\alpha \to 0^+}\int_{(\mathbb{R}^2)^{n+1}}\mathcal{F}_n(x;y_1,\ldots, y_n)\int_0^{\infty}
\Big(\prod_{k=1}^n\mathbbm{g}_{\alpha\gamma_k}(y_k)\Big)
\,\textup{d}s\,\textup{d}x\,\textup{d}y_1\cdots\textup{d}y_n \\
& \quad -2\pi\lim_{\beta \to \infty}\int_{(\mathbb{R}^2)^{n+1}}\mathcal{F}_n(x;y_1,\ldots, y_n)\int_0^{\infty}
\Big(\prod_{k=1}^n\mathbbm{g}_{\beta\gamma_k}(y_k)\Big)
\,\textup{d}s\,\textup{d}x\,\textup{d}y_1\cdots\textup{d}y_n .
\end{align*}
Interpreting the above limits as distributions, we find that the second limit is equal to $0$, while the first is equal to
\[
\int_{\mathbb{R}^2}f(x)^{2^n}\,\textup{d}x=\Vert f \Vert_{\textup{L}^{2^n}(\mathbb{R}^2)}^{2^n}.
\]
{ Thus,
\[ \sum_{m=1}^n\Theta_{\gamma_1, \ldots, \gamma_n}^{n,m}(f) = 2\pi \Vert f \Vert_{\textup{L}^{2^n}(\mathbb{R}^2)}^{2^n} \]
and each of the summands is nonnegative.}
In particular,
\begin{equation}\label{eq:singspec} 
0\leq  \Theta_{\gamma_1, \ldots, \gamma_n}^{n,m}(f) \leq 2\pi \Vert f \Vert_{\textup{L}^{2^n}(\mathbb{R}^2)}^{2^n} 
\end{equation}
for any choice of the parameters.
Using this and summing in $m$ we finally get that \eqref{error to bound} is bounded by
\begin{align*}
&\varepsilon^{-3n}\log\Big(\frac{1}{\varepsilon}\Big) \int_{[1,\infty)^{n-1}}\sum_{m=1}^n\Theta^{n,m}_{ \gamma_1,\ldots,\gamma_{m-1},\sqrt{2}, \gamma_{m+1},\ldots, \gamma_n}(f)
\frac{\textup{d}\gamma_1}{\gamma_1^2} \cdots\,\frac{\textup{d}\gamma_{n-1}}{\gamma_{n-1}^2} \\
&\lesssim\varepsilon^{-3n}\log\Big(\frac{1}{\varepsilon}\Big) \Vert f \Vert_{\textup{L}^{2^n}(\mathbb{R}^2)}^{2^n}
\leq \varepsilon^{-3n}\log\Big(\frac{1}{\varepsilon}\Big)R^{2}.
\end{align*}
This finalizes the proof of \eqref{eq:error}.
\end{proof}


\section{The uniform part}
\label{sec:uniform}

\begin{proposition}
For real numbers $\lambda>0$, $\varepsilon\in (0,1]$, and a measurable set $B\subseteq[0,R]^2$ we have
\begin{equation}\label{eq:uniform}
\big\vert \mathcal{N}^{0}_{\lambda}(\mathbbm{1}_B)-\mathcal{N}^{\varepsilon}_{\lambda}(\mathbbm{1}_B)\big\vert  \leq C_{\textup{uni}} \varepsilon^{1/2} R^2,
\end{equation}
where $C_{\textup{uni}}\in(0,\infty)$ is some constant that depends only on $n$.
\end{proposition}
\begin{proof}
Again write $f=\mathbbm{1}_B$. Motivated by the limit \eqref{eq:Nconvergence}, instead of bounding $\big \vert  \mathcal{N}^{0}_{\lambda}(\mathbbm{1}_B)-\mathcal{N}^{\varepsilon}_{\lambda}(\mathbbm{1}_B)\big\vert $, we rather focus on controlling $\big \vert \mathcal{N}^{\theta}_{\lambda}(\mathbbm{1}_B)-\mathcal{N}^{\varepsilon}_{\lambda}(\mathbbm{1}_B)\big\vert $ for every $0<\theta<\varepsilon\leq 1$. This allows us to reuse the approach from the previous section, which has already proven itself convenient. We begin by bounding $\big\vert  \mathcal{L}_{\lambda}^{\theta,\varepsilon,m}(f)\big\vert $ for a fixed $m \in \{1,\ldots,n\}$. Recall
\begin{align*}
\mathcal{L}_{\lambda}^{\theta,\varepsilon,m}(f)=\frac{-1}{2\pi}\int_\theta^{\varepsilon}\int_{(\mathbb{R}^2)^{n+1}} &\mathcal{F}_{n}(x;y_1,\ldots,y_n)\\
& (\sigma_{\lambda}\ast \mathbbm{k}_{\lambda t})(y_m)\prod_{\substack{1\leq k \leq n \\ k\neq m}}(\sigma_{\lambda}\ast \mathbbm{g}_{t\lambda})(y_k)
\,\textup{d}x\,\textup{d}y_1\cdots\textup{d}y_n \frac{\textup{d}t}{t} .
\end{align*}

Let us focus our attention only on the integrals in the variables $x$ and $y_m$. Expanding the previous display { using \eqref{eq:rekurzija}, we obtain}
\begin{align*}
\int_{(\mathbb{R}^2)^2}
\mathcal{F}_{n-1}(x+y_m;y_1,\ldots, y_{m-1},y_{m+1}, \ldots,y_{n})
\mathcal{F}_{n-1}(x;y_1,\ldots, y_{m-1},y_{m+1}, \ldots,y_{n}) & \\
(\sigma_{\lambda} \ast \mathbbm{k}_{t \lambda})((x+y_m)-x)
\,\textup{d}y_m\,\textup{d}x & .
\end{align*}
Using (by now standard) change of variables $u=x+y_m$ we notice a triple convolution
\[
\int_{\mathbb{R}^2} \big (F_m \ast \sigma_{\lambda} \ast \mathbbm{k}_{t \lambda}\big)(u) F_m(u)\,\textup{d}u.
\]
Applying Plancherel's theorem and the basic identity for the Fourier transform of a convolution, we can write the previous integral as 
\[
\int_{\mathbb{R}^2}
\widehat{F_m}(\xi)\widehat{\sigma_{\lambda}}(\xi)\widehat{\mathbbm{k}_{t \lambda}}(\xi)
\overline{\widehat{F_m}(\xi)}
\,\textup{d}\xi.
\]
Using the decay of the Fourier transform of the spherical measure and the basic identity for the Fourier transform of the Laplacian, we get the following bound
{ 
\begin{align*}
\big \vert  \widehat{\sigma}(\lambda \xi) \widehat{\mathbbm{k}}(t \lambda \xi)\big \vert  &\lesssim
\lambda^{-1/2} \vert \xi\vert ^{-1/2} t^2 \lambda^2 \vert \xi\vert ^2 e^{- \pi t^2 \vert \lambda \vert ^2\vert \xi\vert ^2}
= t^{1/2} (t \lambda \vert \xi\vert )^{3/2}e^{- \pi t^2 \vert \lambda \vert ^2\vert \xi\vert ^2}\\
&\leq t^{1/2} \sup_{u \in [0,\infty)}u^{3/2}e^{-\pi  u^2}
\lesssim t^{1/2}.
\end{align*}
}
Combining these results and using Plancherel's identity once more, we obtain the bound
\begin{align*}
&\big\vert  \mathcal{L}_{\lambda}^{\theta,\varepsilon,m}(f) \big\vert  \lesssim 
\int_{\theta}^{\varepsilon}\int_{(\mathbb{R}^2)^{n}}   t^{1/2}  \big\vert \widehat{F_m}(\xi)\big\vert ^2
\prod_{\substack{1\leq k \leq n \\ k\neq m}}(\sigma_{\lambda}\ast \mathbbm{g}_{t\lambda})(y_k)
\,\textup{d}\xi\,\textup{d}y_1\cdots\textup{d}y_{m-1}\,\textup{d}y_{m+1}\cdots\textup{d}y_n \frac{\textup{d}t}{t} \\
&=\int_{\theta}^{\varepsilon}\int_{(\mathbb{R}^2)^{n-1}}   t^{1/2}  { \big\Vert \widehat{F_m}\big\Vert_{\textup{L}^2(\mathbb{R}^2)}^2}
\prod_{\substack{1\leq k \leq n \\ k\neq m}}(\sigma_{\lambda}\ast \mathbbm{g}_{t\lambda})(y_k)
\,\textup{d}y_1\cdots\textup{d}y_{m-1}\,\textup{d}y_{m+1}\cdots\textup{d}y_n \frac{\textup{d}t}{t} \\
&=\int_{\theta}^{\varepsilon}\int_{(\mathbb{R}^2)^{n-1}}   t^{1/2}  { \left\Vert F_m\right\Vert_{\textup{L}^2(\mathbb{R}^2)}^2}
\prod_{\substack{1\leq k \leq n \\ k\neq m}}(\sigma_{\lambda}\ast \mathbbm{g}_{t\lambda})(y_k)
\,\textup{d}y_1\cdots\textup{d}y_{m-1}\,\textup{d}y_{m+1}\cdots\textup{d}y_n \frac{\textup{d}t}{t}.
\end{align*}

Recall that $f=\mathbbm{1}_{B}$ and $B\subseteq [0,R]^2$, use the trivial bound on  { $\left\Vert F_m\right\Vert_{\textup{L}^2(\mathbb{R}^2)}^2$}, and integrate in all remaining variables except $t$ to obtain
\begin{align*}
\big\vert  \mathcal{L}_{\lambda}^{\theta,\varepsilon,m}(f) \big\vert  &\lesssim 
\int_{\theta}^{\varepsilon}\int_{(\mathbb{R}^2)^{n-1}}   t^{1/2}  R^2
\prod_{\substack{1\leq k \leq n \\ k\neq m}}(\sigma_{\lambda}\ast \mathbbm{g}_{t\lambda})(y_k)
\,\textup{d}y_1\cdots\textup{d}y_{m-1}\,\textup{d}y_{m+1}\cdots\textup{d}y_n \frac{\textup{d}t}{t}\\
&\lesssim
\int_{\theta}^{\varepsilon}   t^{-1/2}  R^2
\,\textup{d}t
\lesssim R^2(\varepsilon^{1/2}-\theta^{1/2}).
\end{align*}
Hence, it follows that
\begin{align*}
\big\vert \mathcal{N}^{\theta}_{\lambda}(\mathbbm{1}_B)-\mathcal{N}^{\varepsilon}_{\lambda}(\mathbbm{1}_B)\big\vert 
\leq
\sum_{m=1}^n\big\vert  \mathcal{L}_{\lambda}^{\theta,\varepsilon,m}(f) \big\vert  &\lesssim
R^2(\varepsilon^{1/2}-\theta^{1/2})\,
\end{align*}
Finally, letting $\theta$ go to $0$ from the right and using \eqref{eq:Nconvergence}, we obtain \eqref{eq:uniform}, as desired.
\end{proof}

\section{Combination of the estimates}

To prove Theorem \ref{thm:large}, we argue by contradiction, i.e., assume the statement does not hold. Then, there exists a strictly increasing sequence of scales $(\lambda_j)_{j \in \mathbb{N}}$ such that $A$ does not contain an isometric copy of $\lambda_j \cdot \Gamma_n$, for all $j\in\mathbb{N}$.
By passing to a subsequence, we can assume that $\lambda_{j+1}\geq 2\lambda_{j}$ for all $j \in \mathbb{N}$. Choose $\varepsilon\in (0 ,1 ]$ small enough so that 
\begin{equation}\label{eq:combaux1}
C_{\textup{uni}}\varepsilon^{1/2}
<\frac{c_{\textup{str}}}{3} \Big(\frac{\overline{\delta}(A)}{2}\Big)^{2^n}.
\end{equation}
Next choose $J \in \mathbb{N}$ large enough so that
\begin{equation}\label{eq:combaux2}
J^{-1}C_{\textup{err}}\varepsilon^{-3n}\log\Big(\frac{1}{\varepsilon}\Big)  
< \frac{c_{\textup{str}}}{3} \Big(\frac{\overline{\delta}(A)}{2}\Big)^{2^n}.
\end{equation}
Finally, take $R\geq1$ such that 
\[
\sup_{x\in\mathbb{R}^2} \frac{\vert A\cap(x+[0,R]^2)\vert }{R^2}>\frac{\overline{\delta}(A)}{2}
\] 
and also make sure that it is large enough, i.e., $R>2\lambda_J$.
It follows that there exists some $x\in \mathbb{R}^2$ such that the set 
\[ B:=(A-x)\cap[0,R]^2 \] 
satisfies 
\begin{equation}\label{eq:combaux3}
\vert B\vert  >\frac{\overline{\delta}(A)}{2} R^2.
\end{equation}
We claim that there has to exists an index $j\in\{1,\ldots,J\}$ such that 
\begin{equation}\label{eq:pigeon_error}    
\big\vert \mathcal{N}^{\varepsilon}_{\lambda_j}(\mathbbm{1}_B)-\mathcal{N}^{1}_{\lambda_j}(\mathbbm{1}_B)\big\vert  \leq J^{-1}C_{\textup{err}}\varepsilon^{-3n}\log\Big(\frac{1}{\varepsilon}\Big) R^{2}.
\end{equation}
If that was not the case, then simply by summing over all $j$, we would obtain a contradiction with \eqref{eq:error}. Combining the bounds \eqref{eq:structured}, \eqref{eq:uniform}, \eqref{eq:combaux1}, \eqref{eq:combaux2}, \eqref{eq:combaux3}, and \eqref{eq:pigeon_error}, and applying the decomposition \eqref{eq:decomposition} with $\lambda=\lambda_j$, we obtain
\[
\mathcal{N}_{\lambda}^0(\mathbbm{1}_B)>\frac{c_{\textup{str}}}{3}\Big(\frac{\overline{\delta}(A)}{2}\Big)^{2^n}R^2>0 .
\]
Therefore, by Remark \ref{rem:nondegenerate}, the set $B$ contains an isometric copy of $\lambda_j \cdot \Gamma_n$.
Since $B$ is obtained simply as a subset of a translate of $A$, it follows that $A$ also contains an isometric copy of $\lambda_j \cdot \Gamma_n$, but this contradicts our choice of the sequence $(\lambda_j)_{j\in\mathbb{N}}$.

To prove Theorem \ref{thm:interval}, consider the intervals $I_j = [2^{-2j},2^{-2j+1})$, where $j\in \mathbb{N}$, and suppose that for each of these intervals we choose some $\lambda_j\in I_j$ such that $A$ does not contain an isometric copy of $\lambda_j\cdot \Gamma_n$. Notice that $\lambda_{j+1}\leq \lambda_j/2$ for all $j\in \mathbb{N}$. We choose $\varepsilon >0$ such that 
\[
C_{\textup{uni}}\varepsilon^{\frac{1}{2}}<\frac{c_{\textup{str}}}{3}\delta^{2^n}
\]
and choose $J$ large enough so that 
\[
J^{-1}C_{\textup{err}}\varepsilon^{-3n}\log\Big(\frac{1}{\varepsilon}\Big) < \frac{c_{\textup{str}}}{3}\delta^{2^n} .
\]
Clearly, such $J$ can be chosen somewhat sparingly, to satisfy
\[ J \leq C \delta^{-(3n+1)2^{n+1}} \]
for a constant $C$ depending on $c_{\textup{str}},C_{\textup{err}},C_{\textup{uni}}$ and thus actually only depending on $n$.
Because at this point we are considering only finitely many scales $\lambda_j$, $1\leq j\leq J$, simply by relabeling them, we can view them as an increasing finite sequence which allows us to reuse the pigeonhole argument from the proof of Theorem \ref{thm:large}, that is, we conclude that for some $j\in \{1,\ldots, J\}$ we have a bound which formally looks identical to \eqref{eq:pigeon_error}, only $B$ is replaced with $A$. 
From the decomposition \eqref{eq:decomposition} we conclude $\mathcal{N}_{\lambda_j}^0(\mathbbm{1}_A)>0$ for some index $j$. By Remark \ref{rem:nondegenerate}, there exists an isometric copy of $\lambda_j\cdot\Gamma_n$ in the set $A$, which contradicts our choices of the numbers $(\lambda_j)_{j\in\mathbb{N}}$.
Thus, there exists an index $1\leq j\leq J$ such that $A$ contains an isometric copy of $\lambda\cdot\Gamma_n$ for every scale $\lambda\in I_j$; set
\[ I(A, n):=I_j \] 
for one such value of $j$.
An easy calculation shows that the length of this interval, 
\[ \vert I(A, n)\vert  = 2^{-2j} \geq 2^{-2J}, \]
satisfies the requirement from the statement of Theorem \ref{thm:interval}. 



\section*{Declarations}
B.\,P. is supported by the \emph{Croatian Science Foundation}.
The research leading to these results received funding from the \emph{Croatian Science Foundation} under project UIP-2017-05-4129 (MUNHANAP).
{ The authors are grateful to Zoran Vondra\v{c}ek and the anonymous referee for several useful comments, which have improved the presentation.}
There is no conflict of interests.
The authors contributed equally to this work.


\bibliography{cubicalgraphs}{}

\begin{thebibliography}{10}

\bibitem{EP}
Thomas Bloom.
\newblock {Erd\H{o}s} problems.
\newblock \url{https://www.erdosproblems.com/}.
\newblock Accessed: September 22, 2023.

\bibitem{B86:roth}
Jean Bourgain.
\newblock A {S}zemer\'{e}di type theorem for sets of positive density in
  {$\mathbb{R}^k$}.
\newblock {\em Israel J. Math.}, 54(3):307--316, 1986.
\newblock \href {https://doi.org/10.1007/BF02764959}
  {\path{doi:10.1007/BF02764959}}.

\bibitem{CMP15:roth}
Brian Cook, \'{A}kos Magyar, and Malabika Pramanik.
\newblock A {R}oth-type theorem for dense subsets of {$\mathbb{R}^d$}.
\newblock {\em Bull. Lond. Math. Soc.}, 49(4):676--689, 2017.
\newblock \href {https://doi.org/10.1112/blms.12043}
  {\path{doi:10.1112/blms.12043}}.

\bibitem{Dur15}
Polona Durcik.
\newblock An {$\textup{L}^4$} estimate for a singular entangled quadrilinear
  form.
\newblock {\em Math. Res. Lett.}, 22(5):1317--1332, 2015.
\newblock \href {https://doi.org/10.4310/MRL.2015.v22.n5.a3}
  {\path{doi:10.4310/MRL.2015.v22.n5.a3}}.

\bibitem{Dur17}
Polona Durcik.
\newblock {$\textup{L}^p$} estimates for a singular entangled quadrilinear
  form.
\newblock {\em Trans. Amer. Math. Soc.}, 369(10):6935--6951, 2017.
\newblock \href {https://doi.org/10.1090/tran/6850}
  {\path{doi:10.1090/tran/6850}}.

\bibitem{DK21}
Polona Durcik and Vjekoslav Kova\v{c}.
\newblock Boxes, extended boxes and sets of positive upper density in the
  {E}uclidean space.
\newblock {\em Math. Proc. Cambridge Philos. Soc.}, 171(3):481--501, 2021.
\newblock \href {https://doi.org/10.1017/S0305004120000316}
  {\path{doi:10.1017/S0305004120000316}}.

\bibitem{DK22}
Polona Durcik and Vjekoslav Kova\v{c}.
\newblock A {S}zemer\'{e}di-type theorem for subsets of the unit cube.
\newblock {\em Anal. PDE}, 15(2):507--549, 2022.
\newblock \href {https://doi.org/10.2140/apde.2022.15.507}
  {\path{doi:10.2140/apde.2022.15.507}}.

\bibitem{DKR18}
Polona Durcik, Vjekoslav Kova\v{c}, and Luka Rimani\'{c}.
\newblock On side lengths of corners in positive density subsets of the
  {E}uclidean space.
\newblock {\em Int. Math. Res. Not. IMRN}, 14(22):6844--6869, 2018.
\newblock \href {https://doi.org/10.1093/imrn/rnx093}
  {\path{doi:10.1093/imrn/rnx093}}.

\bibitem{DKST19}
Polona Durcik, Vjekoslav Kova\v{c}, Kristina~Ana \v{S}kreb, and Christoph
  Thiele.
\newblock Norm variation of ergodic averages with respect to two commuting
  transformations.
\newblock {\em Ergodic Theory Dynam. Systems}, 39(3):658--688, 2019.
\newblock \href {https://doi.org/10.1017/etds.2017.48}
  {\path{doi:10.1017/etds.2017.48}}.

\bibitem{DST22}
Polona Durcik, Lenka Slav\'{\i}kov\'{a}, and Christoph Thiele.
\newblock Local bounds for singular {B}rascamp-{L}ieb forms with cubical
  structure.
\newblock {\em Math. Z.}, 302(4):2375--2405, 2022.
\newblock \href {https://doi.org/10.1007/s00209-022-03148-8}
  {\path{doi:10.1007/s00209-022-03148-8}}.

\bibitem{DS22}
Polona Durcik and Mario Stip\v{c}i\'{c}.
\newblock Quantitative bounds for product of simplices in subsets of the unit
  cube.
\newblock {\em Israel J. Math.}, to appear, 2022.
\newblock Available at: arXiv:2206.10004.

\bibitem{DT20}
Polona Durcik and Christoph Thiele.
\newblock Singular {B}rascamp-{L}ieb inequalities with cubical structure.
\newblock {\em Bull. Lond. Math. Soc.}, 52(2):283--298, 2020.
\newblock \href {https://doi.org/10.1112/blms.12310}
  {\path{doi:10.1112/blms.12310}}.

\bibitem{DT21survey}
Polona Durcik and Christoph Thiele.
\newblock Singular {B}rascamp-{L}ieb: a survey.
\newblock In {\em Geometric aspects of harmonic analysis}, volume~45 of {\em
  Springer INdAM Ser.}, pages 321--349. Springer, Cham, 2021.
\newblock \href {https://doi.org/10.1007/978-3-030-72058-2\_9}
  {\path{doi:10.1007/978-3-030-72058-2\_9}}.

\bibitem{ET12}
Tanja Eisner and Terence Tao.
\newblock Large values of the {G}owers-{H}ost-{K}ra seminorms.
\newblock {\em J. Anal. Math.}, 117:133--186, 2012.
\newblock \href {https://doi.org/10.1007/s11854-012-0018-2}
  {\path{doi:10.1007/s11854-012-0018-2}}.

\bibitem{Erd83:open}
Paul Erd\H{o}s.
\newblock Some combinatorial, geometric and set theoretic problems in measure
  theory.
\newblock In D.~K\"{o}lzow and D.~Maharam-Stone, editors, {\em Measure
  {T}heory, {O}berwolfach 1983: {P}roceedings of the {C}onference held at
  {O}berwolfach, {J}une 26--{J}uly 2, 1983}, volume 1089 of {\em Lecture Notes
  in Mathematics}, pages 321--327. Springer, Berlin, Heidelberg, 1984.
\newblock \href {https://doi.org/10.1007/BFb0072626}
  {\path{doi:10.1007/BFb0072626}}.

\bibitem{EG79}
Paul Erd\H{o}s and Ronald~Lewis Graham.
\newblock Old and new problems and results in combinatorial number theory: van
  der {W}aerden's theorem and related topics.
\newblock {\em Enseign. Math. (2)}, 25(3-4):325--344, 1979.

\bibitem{FKY22}
Kenneth~J. Falconer, Vjekoslav Kova\v{c}, and Alexia Yavicoli.
\newblock The density of sets containing large similar copies of finite sets.
\newblock {\em J. Anal. Math.}, 148(1):339--359, 2022.
\newblock \href {https://doi.org/10.1007/s11854-022-0231-6}
  {\path{doi:10.1007/s11854-022-0231-6}}.

\bibitem{FM86:dist}
Kenneth~J. Falconer and John~M. Marstrand.
\newblock Plane sets with positive density at infinity contain all large
  distances.
\newblock {\em Bull. London Math. Soc.}, 18(5):471--474, 1986.
\newblock \href {https://doi.org/10.1112/blms/18.5.471}
  {\path{doi:10.1112/blms/18.5.471}}.

\bibitem{FIMW21}
David FitzPatrick, Alex Iosevich, Brian McDonald, and Emmett~L. Wyman.
\newblock The {VC}-dimension and point configurations in $\mathbb{F}^2_q$
  (preprint).
\newblock Available at: arXiv:2108.13231, 2021.

\bibitem{FKW90:dist}
Hillel Furstenberg, Yitzchak Katznelson, and Benjamin Weiss.
\newblock Ergodic theory and configurations in sets of positive density.
\newblock In {\em Mathematics of {R}amsey theory}, volume~5 of {\em Algorithms
  Combin.}, pages 184--198. Springer, Berlin, 1990.
\newblock \href {https://doi.org/10.1007/978-3-642-72905-8\_13}
  {\path{doi:10.1007/978-3-642-72905-8\_13}}.

\bibitem{Gow01}
William~Timothy Gowers.
\newblock A new proof of {S}zemer\'{e}di's theorem.
\newblock {\em Geom. Funct. Anal.}, 11(3):465--588, 2001.
\newblock \href {https://doi.org/10.1007/s00039-001-0332-9}
  {\path{doi:10.1007/s00039-001-0332-9}}.

\bibitem{GB15}
Ron Graham and Steve Butler.
\newblock {\em Rudiments of {R}amsey theory}, volume 123 of {\em CBMS Regional
  Conference Series in Mathematics}.
\newblock Conference Board of the Mathematical Sciences, Washington, DC; by the
  American Mathematical Society, Providence, RI, second edition, 2015.
\newblock \href {https://doi.org/10.1090/cbms/123}
  {\path{doi:10.1090/cbms/123}}.

\bibitem{HK05}
Bernard Host and Bryna Kra.
\newblock Nonconventional ergodic averages and nilmanifolds.
\newblock {\em Ann. of Math. (2)}, 161(1):397--488, 2005.
\newblock \href {https://doi.org/10.4007/annals.2005.161.397}
  {\path{doi:10.4007/annals.2005.161.397}}.

\bibitem{Kol04}
Mihalis~N. Kolountzakis.
\newblock Distance sets corresponding to convex bodies.
\newblock {\em Geom. Funct. Anal.}, 14(4):734--744, 2004.
\newblock \href {https://doi.org/10.1007/s00039-004-0472-9}
  {\path{doi:10.1007/s00039-004-0472-9}}.

\bibitem{Kov11}
Vjekoslav Kova\v{c}.
\newblock Bellman function technique for multilinear estimates and an
  application to generalized paraproducts.
\newblock {\em Indiana Univ. Math. J.}, 60(3):813--846, 2011.
\newblock \href {https://doi.org/10.1512/iumj.2011.60.4784}
  {\path{doi:10.1512/iumj.2011.60.4784}}.

\bibitem{Kov12}
Vjekoslav Kova\v{c}.
\newblock Boundedness of the twisted paraproduct.
\newblock {\em Rev. Mat. Iberoam.}, 28(4):1143--1164, 2012.
\newblock \href {https://doi.org/10.4171/RMI/707} {\path{doi:10.4171/RMI/707}}.

\bibitem{Kov16}
Vjekoslav Kova\v{c}.
\newblock Quantitative norm convergence of double ergodic averages associated
  with two commuting group actions.
\newblock {\em Ergodic Theory Dynam. Systems}, 36(3):860--874, 2016.
\newblock \href {https://doi.org/10.1017/etds.2014.87}
  {\path{doi:10.1017/etds.2014.87}}.

\bibitem{K20:anisotrop}
Vjekoslav Kova\v{c}.
\newblock Density theorems for anisotropic point configurations.
\newblock {\em Canad. J. Math.}, 74(5):1244--1276, 2022.
\newblock \href {https://doi.org/10.4153/S0008414X21000225}
  {\path{doi:10.4153/S0008414X21000225}}.

\bibitem{Kovac:survey}
Vjekoslav Kova\v{c}.
\newblock Large copies of large configurations in large sets (extended
  abstract).
\newblock In T.~Orponen, P.~Shmerkin, and H.~Wang, editors, {\em Incidence
  {P}roblems in {H}armonic {A}nalysis, {G}eometric {M}easure {T}heory, and
  {E}rgodic {T}heory. {W}orkshop report 25/2023 of the conference held at
  {O}berwolfach, {J}une 4--{J}une 9, 2023}, volume~20 of {\em Oberwolfach Rep.}
  EMS Press, EMS-Publishing House GmbH, 2023.
\newblock \href {https://doi.org/10.4171/OWR/2023/25}
  {\path{doi:10.4171/OWR/2023/25}}.

\bibitem{Kov23color}
Vjekoslav Kova\v{c}.
\newblock Monochromatic boxes of unit volume (preprint).
\newblock Available at: arXiv:2309.09973, 2023.

\bibitem{KS20}
Vjekoslav Kova\v{c} and Mario Stip\v{c}i\'{c}.
\newblock Convergence of ergodic-martingale paraproducts.
\newblock {\em Statist. Probab. Lett.}, 164:108826, 6, 2020.
\newblock \href {https://doi.org/10.1016/j.spl.2020.108826}
  {\path{doi:10.1016/j.spl.2020.108826}}.

\bibitem{KTZ15}
Vjekoslav Kova\v{c}, Christoph Thiele, and Pavel Zorin-Kranich.
\newblock Dyadic triangular {H}ilbert transform of two general functions and
  one not too general function.
\newblock {\em Forum Math. Sigma}, 3:Paper No. e25, 27, 2015.
\newblock \href {https://doi.org/10.1017/fms.2015.25}
  {\path{doi:10.1017/fms.2015.25}}.

\bibitem{KS15}
Vjekoslav Kova\v{c} and Kristina~Ana \v{S}kreb.
\newblock One modification of the martingale transform and its applications to
  paraproducts and stochastic integrals.
\newblock {\em J. Math. Anal. Appl.}, 426(2):1143--1163, 2015.
\newblock \href {https://doi.org/10.1016/j.jmaa.2015.02.015}
  {\path{doi:10.1016/j.jmaa.2015.02.015}}.

\bibitem{LM16:prod}
Neil Lyall and \'{A}kos Magyar.
\newblock Product of simplices and sets of positive upper density in
  $\mathbb{R}^d$.
\newblock {\em Math. Proc. Cambridge Philos. Soc.}, 165(1):25--51, 2018.
\newblock \href {https://doi.org/10.1017/S0305004117000184}
  {\path{doi:10.1017/S0305004117000184}}.

\bibitem{LM20}
Neil Lyall and \'{A}kos Magyar.
\newblock Distance graphs and sets of positive upper density in $\mathbb{R}^d$.
\newblock {\em Anal. PDE}, 13(3):685--700, 2020.
\newblock \href {https://doi.org/10.2140/apde.2020.13.685}
  {\path{doi:10.2140/apde.2020.13.685}}.

\bibitem{LM19:hypergraphs}
Neil Lyall and \'{A}kos Magyar.
\newblock Weak hypergraph regularity and applications to geometric {R}amsey
  theory.
\newblock {\em Trans. Amer. Math. Soc. Ser. B}, 9:160--207, 2022.
\newblock \href {https://doi.org/10.1090/btran/61}
  {\path{doi:10.1090/btran/61}}.

\bibitem{Sze83}
Laszlo~A. Sz\'{e}kely.
\newblock Remarks on the chromatic number of geometric graphs.
\newblock In {\em Graphs and other combinatorial topics ({P}rague, 1982)},
  volume~59 of {\em Teubner-Texte Math.}, pages 312--315. Teubner, Leipzig,
  1983.

\end{thebibliography}
\bibliographystyle{plainurl}

\end{document}